\documentclass{amsart}
\usepackage[utf8]{inputenc}
\usepackage[T1]{fontenc}
\usepackage[english]{babel}
\usepackage{csquotes}

\usepackage{color} 
\usepackage[colorlinks=true,linktoc=all,linkcolor=blue,citecolor=red,filecolor=pink,urlcolor=blue]{hyperref}
\usepackage[hyperref=true,style=alphabetic,citestyle=alphabetic,backend=bibtex,maxnames=100,doi=true,isbn=false,url=false,giveninits=true]{biblatex}
\bibliography{biblio}

\usepackage{todonotes}

\usepackage{amsmath,amsthm,amssymb}
\theoremstyle{plain}
\newtheorem{theorem}{Theorem}[section]
\newtheorem{lemma}[theorem]{Lemma}
\newtheorem{corollary}[theorem]{Corollary}
\newtheorem{proposition}[theorem]{Proposition}

\theoremstyle{remark}
\newtheorem{remark}[theorem]{Remark}
\newtheorem{example}[theorem]{Example}
\newtheorem{question}{Question}

\newcommand{\ee}{\mathbb{E}}
\newcommand{\cc}{\mathbb{C}}

\newcommand{\hh}{\mathbb{H}}
\newcommand{\rr}{\mathbb{R}}

\newcommand{\nn}{\mathbb{N}}
\newcommand{\qq}{\mathbb{Q}}
\newcommand{\cp}{\mathbb{CP}^1}

\newcommand{\teich}[1]{\mathcal{T}(#1)}
\newcommand{\projteich}[1]{\mathcal{P}(#1)}
\newcommand{\projlocus}[2]{\mathcal C(#1,#2)}
\newcommand{\kmt}[1]{\pi_{#1}}

\newcommand{\pslr}{\operatorname{PSL}_2(\rr)}
\newcommand{\pslc}{\operatorname{PSL}_2(\cc)}

\newcommand{\st}[2]{\operatorname{st}(#1,#2)}
\newcommand{\lk}[2]{\operatorname{lk}(#1,#2)}
\newcommand{\area}[3]{\operatorname{Area}_{#2,#3}(#1)}

\newcommand{\aut}[1]{\operatorname{Aut}(#1)}
\newcommand{\gb}{Galois Bely\u{\i}}

\usepackage{tikz}
\usetikzlibrary{arrows}

\title[Infinite circle packings on surfaces with conical singularities]{Infinite circle packings on surfaces \\ with conical singularities}

\author{Philip L. Bowers}
\address{Department of Mathematics - Florida State University, 1017 Academic Way, Tallahassee, FL 32304, USA}
\email{bowers@math.fsu.edu}

\author{Lorenzo Ruffoni}
\address{Department of Mathematics and Statistics - Binghamton University, Binghamton, NY 13902, USA}
\email{lorenzo.ruffoni2@gmail.com}

\date{\today}

\makeatletter
\@namedef{subjclassname@2020}{\textup{2020} Mathematics Subject Classification}
\makeatother

 \subjclass[2020]{52C26, 30F60}

 \keywords{Circle packing; triangulation; hyperbolic metric; conical singularity; Galois Bely\u{\i} surface.}

\begin{document}

\begin{abstract}
We show that given an infinite triangulation $K$ of a surface with punctures (i.e., with no vertices at the punctures) and a set of target cone angles smaller than $\pi$ at the punctures that satisfy a Gauss-Bonnet inequality, there exists a hyperbolic metric that has the prescribed angles and supports a circle packing in the combinatorics of $K$.
Moreover, if $K$ is very symmetric, then we can identify the underlying Riemann surface and show that it does not depend on the angles.
In particular, this provides examples of a triangulation $K$ and a conformal class $X$ such that there are infinitely many conical hyperbolic structures in the conformal class $X$ with a circle packing in the combinatorics of $K$.
This is in sharp contrast with a conjecture of Kojima-Mizushima-Tan in the closed case.
\end{abstract}

\maketitle

\tableofcontents


\section{Introduction}\label{sec:intro}
\addtocontents{toc}{\protect\setcounter{tocdepth}{1}}
\noindent Let $S$ be a closed connected orientable surface of finite type and $\chi(S)<0$.
We denote by $\teich S$   the Teichm\"uller space of $S$, and by $\projteich S$ the deformation space of complex projective structures on $S$; see \cite{DU09}. 
Let $K$ be a triangulation of $S$, i.e. a simplicial complex whose geometric realization $|K|$ is homeomorphic to $S$.
By the classical Koebe-Andreev-Thurston (KAT) Theorem there is a unique (up to isometry) hyperbolic metric on $S$ that supports a circle packing in the combinatorics of $K$; see \cite{TH97,BO20}.

We denote by $\projlocus SK$ the subspace of $\projteich S$ consisting of projective structures supporting a circle packing in the combinatorics of $K$.
Then the KAT Theorem implies that  $\projlocus SK$ is non-empty.
We consider the natural forgetful map
$$\kmt K:\mathcal C(S,K)\to \teich S .$$
A conjecture by Kojima, Mizushima, and Tan claims that this map is a diffeomorphism when $S$ is a closed surface \cite{KMT03,KMT06,SY18}.

Our goal is to show that this conjecture fails for non-compact surfaces, already in the context of hyperbolic surfaces with conical singularities of amplitude $\theta<\pi$; see Theorem~\ref{thm:failure KMT infinite}.
More precisely, let now $S$ be a non-compact surface of finite type. 
In this context, $\projteich S$ should be taken to be the space of projective structures that have conical singularities or cusps at the ends of $S$, so that it contains the space of hyperbolic metrics with conical singularities or cusps; see \cite{BBCR21} and references therein.
Then, $\projlocus SK$ is defined as in the closed case.
We construct in \S\ref{sec: prescribe conformal} examples of a triangulation $K$ of $S$ and $X\in \teich S$ with a continuum of  structures in $\kmt K^{-1}(X)$, i.e., structures in the conformal class $X$ supporting a circle packing in the combinatorics of $K$.
Hence, the  map $\kmt K$ is very far away from being injective. 
On a thrice-punctured sphere there is only one conformal structure, and this phenomenon occurs for any triangulation.

When $K$ is an ideal triangulation of a punctured surface, one can keep track of the local geometry at the punctures by labelling the ideal vertices.
So, in that case, the aforementioned lack of uniqueness is expected, as one can arbitrarily prescribe the cone angles.
On the other hand, in this paper we are dealing with triangulations of a punctured surface that are genuine infinite triangulations, i.e., punctures are not vertices, but are accumulated by vertices.
In this case the local geometry at a puncture cannot be prescribed by labelling some vertices, and one could expect the local behavior at a puncture to be emergent from global data.
In particular, one could expect some sort of peripheral rigidity for the cone angles.
We show that this is not the case, and one can indeed arbitrarily prescribe some cone angles at the punctures, see Theorem~\ref{thm: main convergence}.

Some interesting questions remain open.
\begin{question}
    We construct examples of $K$ and $X$ for which cone angles can be prescribed in $[0,\pi)$.  Is the set $\pi_K^{-1}(X)$ (pre-)compact?
\end{question}
It should be noted that $\pi_K$ is known to be proper in the context of closed surfaces by work of Schlenker and Yarmola \cite{SY18}.

\begin{question}
In our examples, the structures in $\pi_K^{-1}(X)$ have different cone angles. 
Is the conjecture true if one restricts to the deformation space of structures with fixed cone angles?
\end{question}


 \subsection*{Outline}
 The paper is organized as follows.
In \S\ref{sec:preliminaries} we fix terminology and notation, and prove some general lemmas about the local geometry of circle packings on surfaces.
In \S\ref{sec: compact conical} we prove the existence of a hyperbolic metric with prescribed conical singularities supporting a circle packing in a given combinatorics on a closed surface.
In \S\ref{sec: circle pack infinite} we  prove the same result for infinite triangulations of a non-compact surface of finite type, via a limiting argument.
Finally in \S\ref{sec: prescribe conformal} we provide examples in which the conformal class can be prescribed.

\subsection*{Acknowledgements}
We thank the organizers of the 2021 workshop on The Geometry of Circle Packings in the Thematic Program on Geometric Constraint Systems, Framework Rigidity, and Distance Geometry at the Fields Institute, as well as the  referees for their careful reading and helpful comments.
L. Ruffoni acknowledges partial support by the AMS, the Simons Foundation, and the INDAM-GNSAGA.


\section{Preliminaries}\label{sec:preliminaries}
\addtocontents{toc}{\protect\setcounter{tocdepth}{2}}
\noindent In this section we fix notation and terminology, and prove some lemmas. A standard reference for circle packing theory on constant curvature surfaces is \cite{Ste05} and for ones on complex projective surfaces is \cite{KMT03}.  The natural group of automorphisms of the Riemann sphere $\cp$ is the group of M\"obius transformations $\pslc$.
Since a M\"obius transformation sends circles to circles, the notion of \textit{circle} is well defined on any surface equipped with a geometric structure that is locally modelled on the geometry of the Riemann sphere defined by the group of M\"obius transformations. This includes surfaces endowed with a hyperbolic, Euclidean, or spherical metric, as well as those with a complex projective structure. In the case of the constant curvature metric surfaces, circles are in fact metric circles, i.e., a circle of radius $r>0$ is the set of points whose distance to a fixed point, its center, is the constant $r$ and we consider such a set a circle only if it bounds a closed topological disk, consisting of those points whose distance to the center is at most $r$. We want to emphasize that a circle should always be thought of as the boundary of its \textit{companion disk} even though the disk is rarely mentioned. 

Similarly, a \textit{circle} in a complex projective surface is a curve that develops to a circle in $\cp$ and that is the boundary of a closed topological disk in the surface. More precisely, a \textit{complex projective structure} on a surface $S$ is defined by a pair $(\operatorname{dev},\rho)$, where $\rho:\pi_1(S)\to \pslc$ is a representation and the \textit{developing map} $\operatorname{dev}:\widetilde S\to \cp$ is a $\rho$-equivariant local diffeomorphism defined on the universal cover $\widetilde{S}$. 
Suitable restrictions of $\operatorname{dev}$ provide  local charts on $S$ with values in $\cp$ and change of coordinates in $\pslc$. So a circle in $S$ is a curve $C$ in $S$ that bounds a closed topological disk and for which $\operatorname{dev} (\widetilde{C})$ is a circle in $\mathbb{CP}^{1}$, where $\widetilde{C}$ is a lift of $C$ to $\widetilde{S}$. Another way to say this is that a circle is a curve bounding a closed topological disk whose intersection with any coordinate chart maps into a circle via the chart map. 
Note that there is no well-defined notion of a circle center or a circle radius for a circle in a general complex projective surface, but each circle does come paired with its rarely mentioned companion disk that it bounds.

When $\rho$ takes values in $\pslr$ and $\operatorname{dev}$ takes values in the upper-half plane, the surface $S$  inherits a natural hyperbolic metric.
We will be interested in the case of surfaces of finite type, endowed with hyperbolic metrics for which the punctures have either elliptic or parabolic holonomy (the holonomy is in particular non-trivial). 
In this case, we will extend the definition of circle to allow the companion disk to contain a single puncture. 
Such a circle is the boundary of a once-punctured disk and it is an essential curve on the surface.
Its developed image in $\hh^2$ is a circle invariant under the holonomy around that puncture. This is either parabolic or elliptic, the puncture is a cusp or cone point respectively, and   we will say that the circle is \textit{centered} at that puncture.
These  structures are \textit{tame} in the sense that the developing map extends continuously to the punctures, see \cite{BBCR21}.
So, one could alternatively regard them as projective structures with singularities on a closed surface, and consider the singular points as  marked points instead of punctures. In this way, these circles become circles in the previous definition, i.e., they bound a topological disk.
We refer the reader to \cite{TH97,DU09,BBCR21} for further details.

\subsection{Circle packings on surfaces}\label{sec: background circle packing}
A \textit{circle packing} $\mathcal P$ on a complex projective surface $S$ is a configuration of circles for which the companion disks of any two circles in $\mathcal{P}$ are either disjoint or tangent at a single point. Note that this implies that the interiors of the companion disks of the circles of $\mathcal{P}$ are pairwise disjoint.
Also note that we allow a companion disk to contain one puncture, but not more.
The \textit{nerve} of $\mathcal P$, denoted as $N(\mathcal{P})$, is the simple graph whose vertex set is $\mathcal{P}$ and in which two vertices, or circles, $C_{1}$ and $C_{2}$ are adjacent if and only if the circles are tangent. 

We say that the packing $\mathcal{P}$ \textit{fills} the surface $S$ if each component of the complement of the union of the companion closed disks is a \textit{triangular interstice}, an open topological disk bounded by three arcs of three circles, not containing a puncture of $S$, if there any punctures. 
In case $\mathcal{P}$ fills $S$, $N(\mathcal{P})$ is the $1$-skeleton of the abstract simplicial $2$-complex $K(\mathcal{P})$ whose $2$-simplices are triples of circles that cut out a triangular interstice. When $\mathcal{P}$ fills $S$, $K(\mathcal{P})$ is an abstract triangulation of the surface $S$ and may be realized geometrically in $S$ by an embedding of the nerve $N(\mathcal{P})$ defined as follows:\;for each vertex, a circle of $\mathcal{P}$, choose a point in the interior of its companion disk as a vertex, and for each edge of $N(\mathcal{P})$ choose an arc in $S$ connecting the two vertices and passing through the point of intersection of the corresponding circles, taking care that the interiors of these arcs are pairwise disjoint. This construction provides an embedding of the nerve that cuts out triangular faces whose union is all of $S$. 
There are two facts to notice about this construction. First, when the surface has a constant curvature metric and the vertices are chosen as the centers of the circles and the edges are geodesic arcs connecting centers, we obtain a \textit{geodesic triangulation} in which the faces are metric triangles. Second, the circles of the packing $\mathcal{P}$ accumulate only at the ideal boundary of the surface $S$.

If an abstract simplicial $2$-complex $K$ is isomorphic with $K(\mathcal P)$ for a circle packing $\mathcal{P}$ on $S$,
we will say that $\mathcal P$ is a circle packing on $S$ \textit{in the combinatorics of} $K$. Given an abstract simplicial triangulation of a surface $S$ we may ask if there is a circle packing on $S$ in the combinatorics of $K$. The answer depends on what one means by the phrase \textit{on a surface $S$}. If $S$ is a complex projective surface with or without an invariant constant curvature metric, there usually is no circle packing on $S$ in the combinatorics of $K$. 
In fact, the triangulation $K$ determines which complex projective structures on a topological surface have circle packings in the combinatorics of $K$. The definitive result for hyperbolic surfaces, for instance, is due to Beardon and Stephenson \cite{BS90} when $K$ has bounded geometry, and to Schramm \cite{Sch91} in general, and says that if $K$ is an abstract triangulation of the topological surface $S$ of hyperbolic type, then there is a unique complete hyperbolic metric on $S$ that supports a circle packing in the combinatorics of $K$; see \cite{Ste05} for a full proof. The question of which complex projective structures on the topological surface $S$ support a circle packing in the combinatorics of $K$ is an open question of considerable interest. The present paper studies this question in case the topological surface $S$ is not compact but is of finite type.

Finally, recall that conformal structures and complex structures on surfaces are equivalent data, in the sense that a conformal structure determines a complex structure, and vice versa.  
In the following, we will conflate the two types of structure whenever convenient.
Moreover, each complex structure  determines a unique metric of constant curvature $-1$, $0$, or $1$ (depending on the topology of the underlying surface). On the other hand, there are many other complex projective structures that are compatible with the given complex structure but that are not defined by a constant curvature metric.

\subsection{Triangulations of surfaces}\label{sec:triangulations}
In this paper  $\bar S$ will denote a closed  connected orientable surface of genus $g\geq 0$. 
Given a finite set of points $P\subseteq \bar S$, we will consider the non-compact surface   $S=\bar S \setminus P$.
(Note that all non-compact surfaces in this paper arise in this way and are therefore of finite type.)
The set $P$ will be called the set of \textit{punctures}, and we will consider metrics that have cusps or conical singularities at $P$.
These will often be defined starting from a triangulation.

A triangulation of  the closed surface $\bar S$ will generally be denoted by $\widehat K$, while a triangulation of the non-compact surface $S$  will be denoted by $K$. 
The $m$-skeletons of $\widehat K$ and $K$ will be denote by $\widehat K^{(m)}$ and $K^{(m)}$ respectively, for $m=0,1,2$.

\begin{remark}
   A triangulation $K$ of a non-compact surface $S$ is an infinite simplicial complex, and no vertex of $K$ is in $P$. Rather, vertices of $K$ accumulate to $P$ and any neighborhood of $P$ contains infinitely many triangles from $K$.
In other words, we are not dealing with ideal triangulations of the punctured surface $S$.
\end{remark}

 In the following, we will assume that a triangulation $K$ of a non-compact surface $S$ satisfies the following conditions.
 \begin{enumerate}
 \item $K$ is simplicial.
 \item Every vertex of $K$ has degree at least $3$. 
 \item $K$ has \textit{bounded degree}:\;there is a uniform upper bound on the number of vertices adjacent to a fixed vertex.
 \end{enumerate}
 Note that these conditions are modelled on the properties of the nerve of a circle packing, according to our definition from \S \ref{sec: background circle packing}.

Finally, following \cite{HS95}, we distinguish triangulations $K$ of a non-compact surface $S$ based on the structure of their ends. 
Let $p\in P$ be one of the punctures. We say that the end corresponding to $p$ is \textit{parabolic}, if the family of edge-paths in $K$ starting at some fixed vertex $v\in  K^{(0)}$ 
and diverging to $p$ has infinite \textit{vertex extremal length}.
Otherwise, we say that the end is \textit{hyperbolic}.
We say that $K$ has \textit{parabolic ends} if all ends are parabolic, and that it has \textit{hyperbolic ends} if all ends are hyperbolic.

Now, suppose that $\mathcal P$ is a circle packing on $S$ in the combinatorics of $K$.
If $K$ has parabolic ends, then $\mathcal P$ fills the surface, and each puncture is accumulated by circles of $\mathcal P$.
As an example, if there is a sequence of distinct peripheral edge-loops going into a given end and having constant combinatorial length, then that end must be parabolic. 
This is the case in \S\ref{sec: prescribe conformal} below.

\subsection{Labels for triangulations}\label{sec:labels}
Let $\widehat K$ be a triangulation of a closed surface $\bar S$.
A \textit{label} for $\widehat K$ is a function 
$$\widehat R:\widehat K^{(0)}\to (0,+\infty].$$
Given a label for $\widehat K$, we can construct a hyperbolic metric on $S$ (possibly with cusps and conical singularities) and a circle packing in the combinatorics of $\widehat K$. 
When $\widehat R$ only takes finite values, this is done via the following procedure.
For each triangle $F=\{ v_1,v_2,v_3 \}$ in $\widehat K$ identify $F$ with a triangle of $\hh^2$ whose edges have length $\widehat R(v_1)+ \widehat R(v_2)$, $\widehat R(v_2)+\widehat R(v_3)$, and $\widehat R(v_1)+\widehat R(v_3)$.
Such a triangle is unique up to hyperbolic isometries.
Pasting together these hyperbolic triangles defines a hyperbolic metric on the complement of the set of vertices.
(When $\widehat R$ is infinite at some vertices, a little extra care is needed, and we refer to \cite[\S 2]{B93} for details.)

To understand the completion of this metric, define the  \textit{angle sum function} $\theta_{\widehat R}$ of $ \widehat R$ as follows. 
For each vertex $v$ let $\theta_{\widehat R}(v)$ be the sum of the angles of the triangles having $v$ as a vertex.
Then $v$ is a cusp when $\theta_{\widehat R}(v)=0$, or equivalently $\widehat R(v)=+\infty$, and a cone point of angle $\theta_{\widehat R}(v)$ otherwise.
In particular, if $\widehat R$ is bounded then one gets a hyperbolic metric with conical singularities.
In order to simplify the exposition, we will also consider a cusp as a cone point with angle $0$.

Finally, if one takes the geodesic circle of radius $\widehat R(v)$ at each vertex with respect to this metric, then one obtains a circle packing on $\bar S$ in the combinatorics of $\widehat K$.
If a circle is centered at a cusp, then it develops to a horocycle in $\hh^2$. If a circle is centered at a smooth point or a conical point, then it develops to a geodesic circle in $\hh^2$.

\begin{figure}[h]
    \centering
    \includegraphics[width=.75\textwidth]{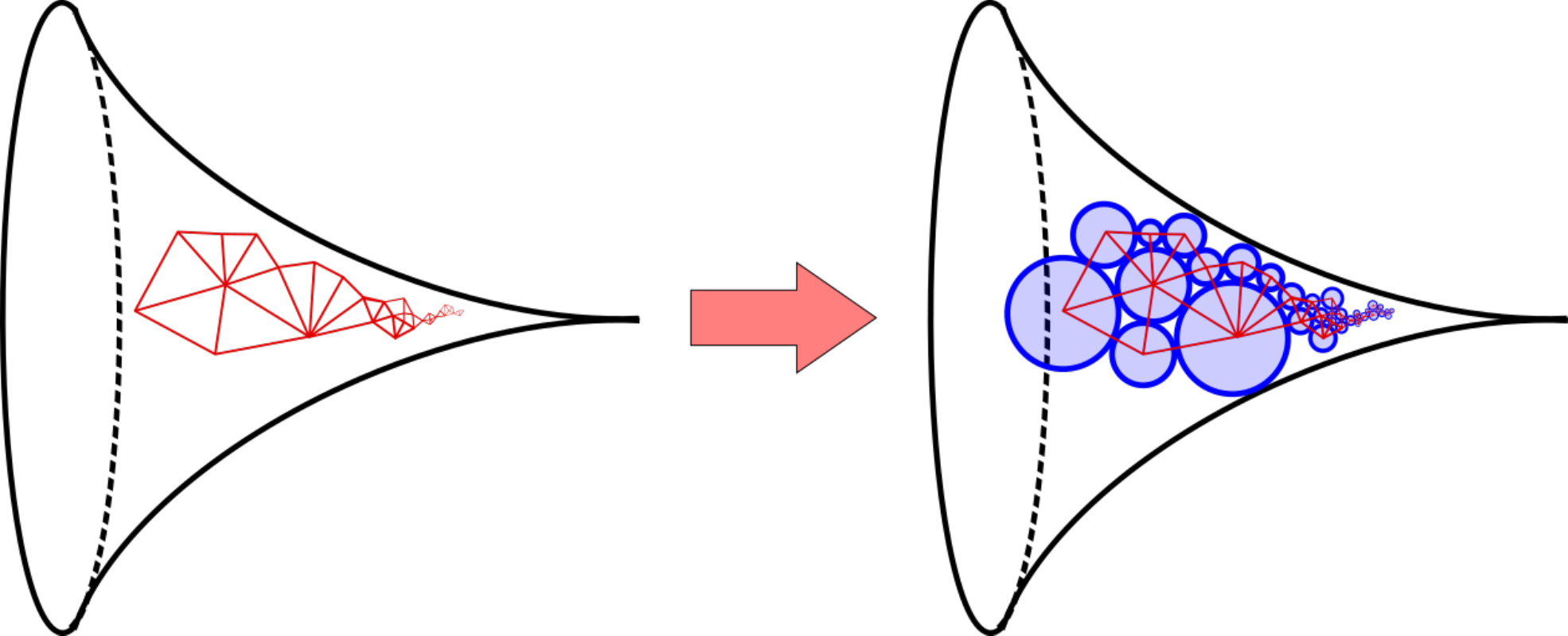}
    \caption{Circle packing an infinite triangulation around a puncture.}
    \label{fig:packing}
\end{figure}

A similar procedure also works for a triangulation $K$ of a non-compact surface $S$ of finite type.
Given a label for $K$ that takes only finite values, one obtains a hyperbolic metric on $S$ supporting a circle packing in the combinatorics of $K$.
However, in this case it is not clear what the completion of this metric looks like.
The main difference is that the punctures of $S$ are not vertices of $K$, so it is not possible to keep track of the geometry at a puncture by means of a label or angle sum function.
So, in a certain sense the peripheral geometry has to be emergent behavior, a priori determined by the global structure of the circle packing.
Our main technical contribution is to show how to prescribe cone angles also in this case, see \S\ref{sec: circle pack infinite}.
Note that in this case, the circle packing has no circles centered at punctures. Rather, circles accumulate into the punctures; see Figure~\ref{fig:packing}.


\subsection{Flowers in a circle packing}
We conclude with some technical lemmas about the interaction between local geometry and local combinatorics in a circle packing.

A \textit{flower} is a collection of circles $\mathcal F = \{C, C_1, \dots, C_k\}$ such that $C$ is tangent to all the other ones, and for $i=1,\dots, k$ the circle $C_i$ is tangent to $C,C_{i-1},C_{i+1}$, with the subscript taken modulo $k$.
The circle $C$ is the \textit{center} and the circles $C_1\dots,C_k$ are the \textit{petals} of the flower. 
See Figure~\ref{fig:circlepacking}.
The integer $k$ is called the \textit{degree} of the flower, and satisfies $k\geq 3$ because $K$ is simplicial.
Given a circle $C$ in a circle packing $\mathcal P$, the collection of circles adjacent to $C$ forms a flower. Moreover, if the circle packing $\mathcal P$ has degree bounded by $d$, then every flower in $\mathcal P$ has degree at most $d$.

\begin{figure}[h]
    \centering
    \includegraphics[width=.75\textwidth]{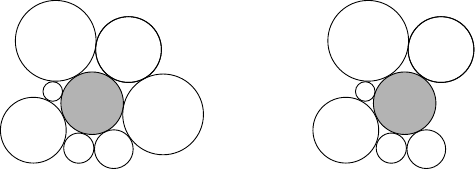}
    \caption{A flower of degree $7$ (left) and a partial flower of degree $6$ (right). The center is the circle bounding the shaded disk. The petals are the surrounding circles.}
    \label{fig:circlepacking}
\end{figure}

Note that these definitions make sense on any surface equipped with a complex projective structure, possibly with cones and cusps.
In the particular case that the projective structure arises from a Euclidean or hyperbolic metric, we are interested in bounds on the radius of a circle that appears in a flower in terms of the degree and the radii of the other circles.

We start with the following upper bound, which is Lemma 5 in \cite{BS91}. 
 The exact optimal value for the constant is $H_k=\ln(\sin(\pi/k)^{-1})$, which correspond to a flower in which all petals are horocyclic.
\begin{lemma}\label{lem: H ring lemma}
Let $C,C_1\dots,C_k$ form a flower of center $C$ and degree $k$ in $\mathbb H^2$.
Let $R_0$ be the (hyperbolic) radius of $C$.
Then there exists a constant $H_k>0$ depending only on $k$ such that
$R_0\leq H_k$.
Moreover, the constants can be chosen such that if $k_1<k_2$ then $H_{k_1}<H_{k_2}$.
\end{lemma}

The next statement is the classical Ring Lemma from \cite{RS87}.
The exact value of the constant $E_k$ has been obtained in \cite{HA88}.
We will not need the exact value, but only the fact that $0<E_k<1$. 

\begin{lemma}[Euclidean Ring Lemma]\label{lem: E ring lemma}
Let $C,C_1\dots,C_k$ form a flower of center $C$ and degree $k$ in $\ee^2$.
Let $r$ be the (Euclidean) radius of $C$ and let $r_i$ be the (Euclidean) radius of $C_i$ for $i=1,\dots,k$.
Then there exists a constant $0<E_k<1$ depending only on $k$ such that $r_i\geq E_k r$ for each $i=1,\dots,k$.
 Moreover, the constants can be chosen such that if $k_1<k_2$ then $E_{k_1}>E_{k_2}$.
\end{lemma}

Given a hyperbolic surface without conical singularities, the developing map embeds the universal cover isometrically into $\hh^2$.
If one also has a circle packing on the surface, then circles centered at regular points develop to geodesic circles, and circles centered at cusps develop to horocycles.
Moreover, a flower centered at a regular point develops injectively to a flower in $\hh^2$ of the same degree, and a flower centered at a cusp develops to a configuration of circles containing a horocycle tangent to all other circles.

On the other hand, for a hyperbolic surface  with  conical singularities the developing map is not an embedding.
Moreover, any flower  for which the center contains a point with cone angle $\theta \in  (0,2\pi)$ will fail to develop to a flower in $\hh^2$, as the developed images of the petals will not close up. 

To address this phenomenon, we introduce the following terminology.
A \textit{partial flower} $\mathcal F_0$ is a configuration of circles obtained from a flower $\mathcal F$ by removing some consecutive petals. 
The number of petals that are left is called the \textit{degree} of the partial flower.
See Figure~\ref{fig:circlepacking}.
In this case, we say that $\mathcal F_0$ can be \textit{completed} to $\mathcal F$, in the sense that it is possible to add  circles to $\mathcal F_0$ to recover $\mathcal F$.
Note that, as an abstract configuration of circles (i.e., not part of a packing), the same partial flower can be completed to different flowers by adding different collections of circles to it.

\begin{lemma}\label{lem: flower at cone}
Let $\bar S$ be a finite area hyperbolic surface (possibly with finitely many cusps and cone angles).
Let $\mathcal P$ be a circle packing on $\bar S$ (possibly with some circles centered at cusps and cones). 
Let $\mathcal F=\{C,C_1,\dots, C_k\} \subseteq \mathcal P$ be a flower of degree $k$ with $C$ centered at a   point of angle $\theta\in (0, 2\pi]$.
Then $\mathcal F$ develops to a partial flower in $\hh^2$ that can be completed to a flower of degree at most $d_{k,\theta}=k\left( \lfloor \frac{2\pi }{\theta} \rfloor +1\right)$.
\end{lemma}
\begin{proof}
    Let $S$ be the surface obtained from $\bar S$ by removing all cusps and cones.
    Lift the packing $\mathcal P$ to a packing of the universal cover of $S$, and then develop to a collection of circles in $\hh^2$.
    The flower $\mathcal F$ develops to a partial flower $\mathcal F_0'$ in $\hh^2$ of degree $k$.
    If $\theta=2\pi$ then
    this is actually already a flower.
    Otherwise, if $0<\theta < 2\pi$,  then we can rotate the petals (i.e., continue the development) around the center up to $m=\lfloor \frac{2\pi}{\theta} \rfloor$ times to obtain a partial flower $\mathcal F_0'$ of degree $k\lfloor \frac{2\pi }{\theta} \rfloor$ (or a flower, if $\theta=\frac{2\pi}{m}$).
    But if we rotate $m+1$ times, then we have an overlap.
    This means that $\mathcal F_0'$ can be completed to a flower $\mathcal F'$ in $\hh^2$ of degree at most $k\left( \lfloor \frac{2\pi }{\theta} \rfloor +1\right)$.
\end{proof}

\begin{remark}
    A similar statement can be obtained for $\theta > 2\pi$. This will not be needed in the present paper.
\end{remark}

\begin{proposition}\label{prop: radii estimates}
Let $\bar S$ be a finite area hyperbolic surface (possibly with finitely many cusps and cone angles $\Theta$ in $(0,2\pi)$).
Let $\mathcal P$ be a circle packing on $\bar S$ (possibly with some circles centered at cusps and cones). 
Assume that $\mathcal P$ has degree bounded by $d$.
Then there exist  constants $H_{B}>0$ and $0<E_D<1$ depending only on $d$ and $\Theta$ such that the following holds.
\begin{enumerate}
    \item \label{item:upper} If $C \in \mathcal P$ is a circle not centered at a cusp, then its (hyperbolic) radius is at most $H_{B}$.

    \item \label{item:comparable} If $C_1,C_2 \in \mathcal P$ are tangent circles not centered at cusps, with (hyperbolic) radii $R_1,R_2$, then we have 
    $$E_D\leq \frac{R_1}{R_2}\leq \frac{1}{E_D}.$$
\end{enumerate}
\end{proposition}
\proof 
As in the previous proof, let $S$ be the surface obtained from $\bar S$ by removing all cusps and cones.
Lift the packing $\mathcal P$ to a packing of the universal cover of $S$, and then develop to a collection of circles in $\hh^2$.

Let $D=D(d,\Theta)$ be the minimum and let $B=B(d,\Theta)$ be the maximum of the constants $d_{k,\theta}$ from Lemma~\ref{lem: flower at cone} over all circles, where $k$ is the degree of the flower and $\theta$ the cone angle.
Since the packing has degree bounded by $d$ and there are only finitely many cone angles, this is well-defined.

Statement \eqref{item:upper} is obtained  as follows.
Let $C\in \mathcal P$ be a circle not centered at a cusp.
If $k$ is the degree of the flower of $C$ and $\theta$ is the angle at its center, then by Lemma~\ref{lem: flower at cone} we can develop it to a partial flower in $\hh^2$ that can be completed to a flower of degree at most $d_{k,\theta}\leq B$. 
Then by Lemma~\ref{lem: H ring lemma} we have that the radius of $C$ is at most $H_{B}$.

To prove  statement \eqref{item:comparable}, let $C_1,C_2 \in \mathcal P$ be two tangent circles not centered at cusps, and let their (hyperbolic) radii be $R_1,R_2$.
Let $\mathcal F_1$ be the flower centered at $C_1$, and note that $C_2$ is a petal. Let $k_1$ be the degree of $\mathcal F_1$.
Let $p$ be the center of $C_1$ and let $\theta_1\leq 2\pi$ be the cone angle at $p$.
By Lemma~\ref{lem: flower at cone} we can develop it to a partial flower in $\hh^2$ that can be completed to a flower $\mathcal F_1'$ of degree at most $d_1=d_{k_1,\theta_1}$.
 Realize $\hh^2$ in the unit disk.
Since hyperbolic circles are Euclidean circles,  we  can regard $\mathcal F_1'$ as a flower in $\ee^2$.

From Lemma~\ref{lem: E ring lemma} we obtain the following relations between their Euclidean radii: $r_2\geq E_{d_1} r_1\geq E_D r_1.$
Since the Euclidean and hyperbolic metric on the unit disk are comparable, it turns out that the same inequalities hold for the hyperbolic radii $r_1,r_2$.
This can be seen via the following computations.

Without loss of generality, we can assume that $C_1$ sits at the origin of the unit disk, and $C_2$ is on its right, with centers sitting on the same horizontal diameter.
The hyperbolic radius of $C_1$ is $R_1= 2\operatorname{arctanh}(r_1)$
and the hyperbolic radius of $C_2$ is  $R_2= 2\operatorname{arctanh}(r_1+r_2) -R_1.$
Since $r_2\geq E_D r_1$, we obtain that
$$R_2\geq  2\operatorname{arctanh}((E_D+1)r_1) -R_1.$$
Using the power series representation of $\operatorname{arctanh}(x)$ at $x=0$, and noting that $E_D+1>1$ and $(E_D+1)r_1\leq r_1+r_2<1$, we obtain that

$$R_2\geq 2\operatorname{arctanh}((E_D+1)r_1) -R_1 \geq $$
$$\geq 2(E_D+1)\operatorname{arctanh}(r_1) -R_1 = (E_D+1)R_1-R_1 =E_DR_1.$$

This shows that $R_2\geq E_D R_1$.
Swapping the roles of $C_1$ and $C_2$ we also obtain that 
$R_1\geq E_D R_2.$
Combining the two inequalities,  one obtains the desired statement.
\endproof


%
\section{Circle packing a finite triangulation with cones}\label{sec: compact conical}
\noindent Let $\bar S$ be a closed connected orientable surface of genus $g\geq 0$, $P=\{p_1,\dots,p_n\}\subseteq \bar S$ a finite set of points, and $S=\bar S \setminus P$. 
Let $\Theta=(\theta_1,\dots,\theta_n)\in [0,+\infty)^n$ be such that

\begin{equation}\label{eq: orbchar}
    2\pi\chi(\bar S) + \sum_{i=1}^n (\theta_i-2\pi) <0.
\end{equation}

\begin{remark}[Uniformization with conical singularities]\label{rem: conical unif}
In every conformal class $X\in \teich {\bar S}$ there is a unique (possibly singular) hyperbolic metric $h_{X,\Theta}$ with cone angle $\theta_i$ on $p_i$ for all $i$; see \cite{MO88,TR91}. 
The case in which for all $i=1,\dots,n$ we have $\theta_i=2\pi$ corresponds to the (smooth) hyperbolic metric that uniformizes $X$.
The case in which for all $i=1,\dots,n$ we have $\theta_i=0$ corresponds to the (cusped) hyperbolic metric of finite volume that uniformizes the punctured Riemann surface $X\setminus P$.
\end{remark}

The goal of this section is to prove Theorem~\ref{thm:combinatorial conical uniformization}, which is a combinatorial analogue to the conical uniformization results mentioned in Remark~\ref{rem: conical unif}, in the sense that here we are prescribing a triangulation instead of prescribing a conformal class. This result is similar to Theorem 4.1 of \cite{B93} in the case where the boundary is empty and replaces condition (i) of that theorem with the more easily checked condition \eqref{eq: orbchar} of the present paper.
Note that we assume cone angles are strictly less than $\pi$.

\begin{theorem}\label{thm:combinatorial conical uniformization}
    Let $\Theta=(\theta_1,\dots,\theta_n)\in [0,\pi)^n$ satisfy \eqref{eq: orbchar}, and let $\widehat K$ be a triangulation of $\bar S$ with $p_i\in \widehat K^{(0)}$.
    Then there exists a hyperbolic metric $h_{\widehat K,\Theta}$ on $\bar S$ with  cone angle $\theta_i$ at $p_i$ for all $i=1,\dots,n$, which supports a circle packing $\mathcal P_{\widehat K,\Theta}$ in the combinatorics of $\widehat K$.
    Moreover, the metric and the packing are uniquely determined up to isometry.
\end{theorem}

\begin{remark}
    The metrics $h_{\widehat K,\Theta}$ and $ h_{X,\Theta}$ are a priori unrelated.
    In particular, it is not clear whether $ h_{X,\Theta}$ supports a circle packing in the combinatorics of $\widehat K$.
    However, if the conformal structure $X_{\widehat K}$ underlying $h_{\widehat K,\Theta}$ happens to coincide with $X$, then the two metrics are the same.
    See \S\ref{sec: prescribe conformal} for examples in which the complex structure can be prescribed and the two metrics coincide.
\end{remark}
 
In order to prove Theorem~\ref{thm:combinatorial conical uniformization}, we need to establish a combinatorial relationship between the triangulation and the cone angles.
(Note that throughout the paper we are assuming that triangulations are simplicial. This is especially relevant in the next two proofs.)
Let $\widehat K$ be any triangulation of the closed surface $\bar S$, with $p_i\in \widehat K^{(0)}$.
The choice of $\Theta$ gives rise to a natural labelling of the vertices of $K$, which can be encoded in the function 
$$\phi_\Theta:\widehat K^{(0)}\to [0,+\infty), \quad \phi_\Theta(v)=  \left\lbrace  \begin{array}{ll}
 \theta_i    & \textrm{ if } v=p_i  \\
  2\pi   &  \textrm{ if } v\neq p_i
\end{array} \right.$$

Given a collection of vertices $V\subseteq \widehat K^{(0)}$ let us consider $\mathcal F_V=\{ F\in \widehat K^{(2)} \ | \ F^{(0)}\cap V \neq \varnothing\}$, i.e., the collection of $2$-faces of $\widehat K$ with at least a vertex in $V$. 
 Let $F_V=|\mathcal F_V|$ be the number of such faces, and let $X_V=\cup_{F\in \mathcal F_V}F$ the subcomplex of $\widehat K$ induced by them.

\begin{lemma}\label{lem: combinatorial invariant}
In the above notation, if $\theta_i<\pi$ for all $i=1,\dots,n$ and $V\subseteq \widehat K^{(0)}$,
then we have that 
$$\pi F_V- \sum_{v\in V}\phi_\Theta(v)>0.$$
\end{lemma}
\proof 
Let $X$ be a conformal class on $\bar S$ and let $h_{X,\Theta}$ be the unique hyperbolic metric of angle $\theta_i$ at $p_i$ in this conformal class $X$; see Remark~\ref{rem: conical unif}.  
We claim that up to isotopy one can assume that $\widehat K$ is a geodesic triangulation with respect to this metric. 
To see this, isotope every edge to the geodesic arc between its endpoints. 
Since  $\theta_i<\pi$, this geodesic  is simple and avoids cone points (possibly except at its endpoints); see \cite[Theorem 5.1]{TWZ06}.
Every triangle of $K$ gets pulled to a geodesic triangle.
We need to check that no triangles  overlap or degenerate.
If there are no singular points, this was done in \cite{dV91}, and the arguments therein can be extended to our case as follows.
To rule out the possibility of overlaps, one uses the Gauss-Bonnet formula for conical metrics to show that, for each vertex $v$, the value of $\phi_\Theta(v)$ and the sum of the angles around $v$ must coincide (if a triangle $F$ degenerates, then the angle at $v$ in $F$ is $0$ or $\pi$); compare  \cite[\S2.a, \S2.d]{dV91}.
Next, since cone angles are less than $\pi$,  if two adjacent triangles degenerate, then they are aligned.
As in \cite[\S2.c]{dV91}, it follows that a maximal subcomplex of $K$ consisting of degenerating triangles is either a disk that degenerates to a geodesic segment, or an annulus that degenerates to a closed geodesic. But  the Gauss-Bonnet formula shows that both cases are absurd.

Now, if $\widehat K$ is a geodesic triangulation with respect to the metric $h_{X,\Theta}$, then the area of the subcomplex $X_V$ with respect to this metric can be computed as
\begin{equation}
    \area{X_V}{X}{\Theta} =\sum_{F\in \mathcal F_V}\area{F}{X}{\Theta}.
\end{equation}

Since $h_{X,\Theta}$ is a hyperbolic metric, each face $F$ is isometric to a hyperbolic triangle, and its area is given by
\begin{equation}
    \area{F}{X}{\Theta}=\pi-\sum_{v\in F^{(0)}}\theta_F(v),
\end{equation}
where $\theta_F(v)$ is the angle at the vertex $v$ inside $F$ with respect to the metric $h_{X,\Theta}$.
(Note that the value of $\theta_F(v)$ depends on the chosen $X$.)
It follows that
\begin{equation}
    0< \area{X_V}{X}{\Theta} = \sum_{F\in \mathcal F_V}  \left(\pi-\sum_{v\in F^{(0)}}\theta_F(v)\right ) = \pi F_V - \sum_{F\in \mathcal F_V} \sum_{v\in F^{(0)}} \theta_F(v).
\end{equation}

Now notice that if $v\in V$, then by definition $v$ is a vertex of a face of $\mathcal F_V$, and actually all the faces of $\st{v}{\widehat K}$ are in $\mathcal F_V$.
 So, we have that

\begin{equation}
    \sum_{F\in \mathcal F_V} \sum_{v\in F^{(0)}} \theta_F(v) \geq \sum_{v\in V} \sum_{F\in \st{v}{\widehat K}} \theta_F(v).
\end{equation}

Since the star of $v$ in $\widehat K$ provides a full neighborhood of $v$ in $\bar S$, the term
$\sum_{F\in \st{v}{\widehat K}} \theta_F(v)$  is just the total angle at $v$ with respect to the metric $h_{X,\Theta}$.
By definition of this metric, the total angle at $v$ is either $2\pi$ if $v$ is a smooth point, or $\theta_i$ if $v=p_i$.
This coincides with the definition of $\phi_\Theta$, so
\begin{equation}
    \sum_{F\in \st{v}{\widehat K}} \theta_F(v) = \phi_\Theta (v).
\end{equation}

By putting everything together we obtain
\begin{equation}
    0< \area{X_V}{X}{\Theta} =   \pi F_V - \sum_{F\in \mathcal F_V} \sum_{v\in F^{(0)}} \theta_F(v) \leq
    \end{equation}
\begin{equation*}
    \leq  \pi F_V - \sum_{v\in V} \sum_{F\in \st{v}{\widehat K}} \theta_F(v) = \pi F_V -\sum_{v\in V}\phi_\Theta (v).
\end{equation*}
\endproof

\begin{proof}[Proof of Theorem~\ref{thm:combinatorial conical uniformization}]
Let $V\subseteq \widehat K^{(0)}$ be a collection of vertices that spans an edge-path connected subcomplex of $\widehat K$.
It follows from  Lemma~\ref{lem: combinatorial invariant} that 
$$\pi F_V- \sum_{v\in V}\phi_\Theta(v)>0.$$
By \cite[Theorem 4.1]{B93} there exists a uniquely determined label $\widehat R:\widehat K^{(0)}\to (0,+\infty ]$ whose associated angle sum function is  $\theta_{\widehat R}=\phi_\Theta$.
For any vertex $v\in \widehat K^{(0)}\setminus P$ we have $\phi_\Theta(v) =2\pi \neq 0$, and therefore $\widehat R(v)<+\infty$.
In particular, via the construction in \S\ref{sec:labels}, the label $\widehat R$ provides a hyperbolic metric on $\bar S$ with a cone of angle $\theta_i=\phi_\Theta(p_i)$ at $p_i$, as well as a circle packing on $\bar S$ in the combinatorics of $\widehat K$. (Recall that by a cone of angle $\theta_i=0$ we mean a cusp.)
\end{proof}

\begin{remark}[Invariance under group action]\label{rem:invariance}
    Suppose a finite group $G$ acts on the surface $S$ by simplicial automorphisms of $\widehat K$ and by preserving the angle vector $\Theta$, in the sense that if $gp_i=p_j$ for some $g\in G$, then $\theta_i=\theta_j$.
    It follows from the uniqueness of the label $\widehat R$ that this label is $G$-invariant.
    Namely, $\widehat R$ must coincide with its average
    $\widehat R_G(v)=\frac{1}{|G|}\sum_{g\in G} \widehat R(gv)$
    under the action of $G$.
    Therefore, the associated hyperbolic metric $h_{\widehat K,\Theta}$ in Theorem~\ref{thm:combinatorial conical uniformization} is $G$-invariant, i.e., $G$ acts by isometries of $h_{\widehat K,\Theta}$.
\end{remark}


\section{Circle packing an infinite triangulation with cones}\label{sec: circle pack infinite}
\noindent Let $K$ be a triangulation of the non-compact surface $S=\bar S\setminus P$.
A \textit{peripheral system for $K$ with respect to $P$} is a collection $C=\{c_1,\dots,c_n\}$ of disjoint simple peripheral edge-loops $c_i$ in $K^{(1)}$ such that $c_i$ encloses $p_i$.
Let $D_i$ be the open punctured disk around $p_i$  bounded by $c_i$. 
Given a peripheral system $C$ for $K$ with respect to $P$, we define a triangulation of the compact surface $\bar S$, called the \textit{coned-off triangulation $\widehat K(C)$ relative to $C$}, as follows.
The triangulation  $\widehat K(C)$ agrees with $K$ outside the disks $D_i$.
Inside a disk $D_i$ we modify $K$ by  collapsing all vertices contained in $D_i$ to $p_i$, as shown in Figure~\ref{fig:coneoff}.
(Notice that $\widehat K(C)$ can also be seen as a triangulation of the non-compact surface $S$ with some ideal vertices at the points of $P$.)

\begin{figure}[h]
    \centering
    \includegraphics[width=.75\textwidth]{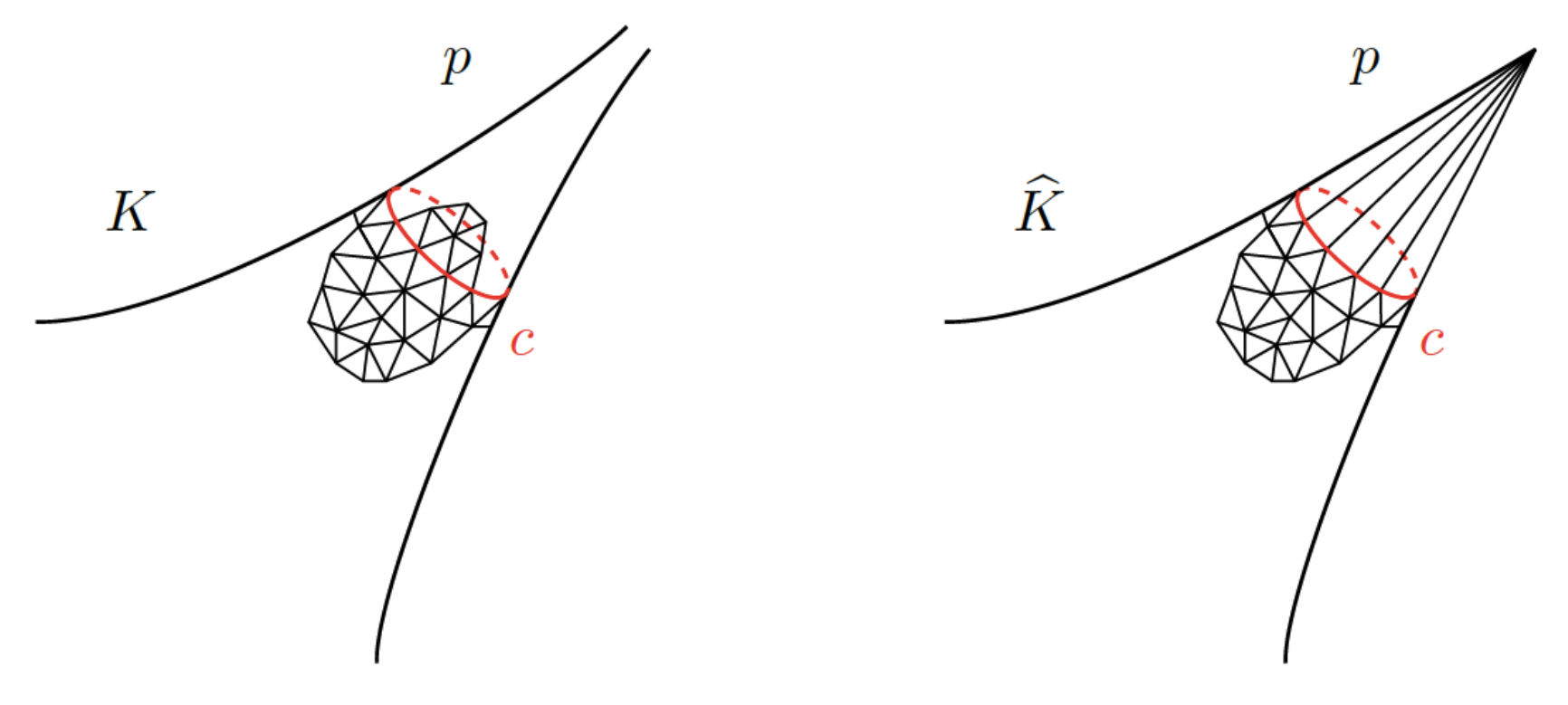}
    \caption{The coned-off triangulation around a puncture.}
    \label{fig:coneoff}
\end{figure} 

Since $\Theta$  satisfies \eqref{eq: orbchar}, by Theorem~\ref{thm:combinatorial conical uniformization} there exist a unique hyperbolic metric $h_{\widehat K(C),\Theta}$ on $\bar S$ with cone points of angle $\theta_i$ on $p_i$, and also a unique circle packing $\mathcal P_{\widehat K(C),\Theta}$ in the combinatorics of $\widehat K(C)$, with $n$ circles centered at the points $p_i$. 

Now, let us fix an exhausting sequence of peripheral systems $C^m$, i.e. $C^m=\{c_1^m,\dots,c_n^m\}$, $c^m_i\underset{m\to +\infty}{\longrightarrow} p_i$ for $i=1,\dots,n$.
This means that if $D_i^m$ denotes the open punctured disk around $p_i$ bounded by $c_i^m$, then for every open  neighborhood $U$ of $p_i$ in $\bar S$ there exists $m\in \nn$ such that $D_i^m\subseteq U$.
Let $\widehat K_m=\widehat K(C^m)$ be the coned-off triangulation relative to $C^m$.
We denote by $h_m=h_{\widehat K_m,\Theta}$  the conical hyperbolic metric
and by  $\mathcal P_m=\mathcal P_{\widehat K_m,\Theta}$ the circle packing 
determined by Theorem~\ref{thm:combinatorial conical uniformization}.
The goal of this section is to prove the following statement.
It is the analogue of Theorem~\ref{thm:combinatorial conical uniformization} in the case of an infinite triangulation. 
Note that here we restrict to angles that are smaller than $\pi$.

\begin{theorem}\label{thm: main convergence}
Let $\Theta=(\theta_1,\dots,\theta_n)\in [0,\pi)^n$ satisfy \eqref{eq: orbchar}.
Let $K$ be a triangulation of $S=\bar S\setminus P$ of degree bounded by $d$. 
Then (up to passing to subsequences) the sequence of conical hyperbolic metrics $h_m$ converges to a hyperbolic metric   $h_{K,\Theta} $ on $\bar S$ with  cone angle $\theta_i$ at $p_i$ for all $i=1,\dots,n$, which supports a circle packing $\mathcal P_{K,\Theta} $ of $S$ in the combinatorics of $K$. 
\end{theorem}

Note that the packing fills the surface if $K$ has parabolic ends; see \S\ref{sec:triangulations}.

\begin{remark}[Thrice-punctured sphere]\label{rem: 3p sphere conformal}
When $S$ is a thrice-punctured sphere, there is a unique conformal structure, and the sequence of metrics is actually constant, since the cone angles are fixed.
In this case there is no need to assume $\theta_i<\pi$.
(Any of these hyperbolic metrics is obtained by a pillow-case construction on a hyperbolic triangular membrane; see for instance \cite{BBCR21}. This is consistent with the fact that  \eqref{eq: orbchar} implies that $\theta_1+\theta_2+\theta_3 <2\pi$.)
One still needs to prove that this metric admits a circle packing in the combinatorics of $K$. 
This does not require $\theta_i<\pi$ either.
Indeed, as the following discussion will show, as soon as one knows that the sequence of metrics converges, the limit metric will support a suitable circle packing essentially by construction: it is defined by a label that is a limit of labels.
\end{remark}

\begin{remark}[Convergence]\label{rem:convergence}
    The proof of Theorem~\ref{thm: main convergence} that we provide below makes uses of a compactness argument, hence the need to extract a subsequence. 
    A priori, it is possible that the limit is not unique.
    If one could show that the sequence of underlying conformal classes is Cauchy, then one would get a unique conformal class in the limit, hence a unique limiting metric by Remark~\ref{rem: conical unif}.    
    In \S\ref{sec: prescribe conformal} we show how to prescribe the conformal class, and get uniqueness, when $K$ and $\Theta$ are invariant under a sufficiently large group of automorphisms (e.g, in the  case of Hurwitz surfaces).
    We conjecture that uniqueness should hold in general, as long as $K$ has parabolic ends.
\end{remark}

The proof of Theorem~\ref{thm: main convergence} in the general case is based on a ring lemma for circle packings from \cite{RS87,HR93} and a collar lemma for hyperbolic surfaces with small cones from \cite{DP07}.

We start by considering the following.
Let $K_m$ be the simplicial complex obtained from $\widehat K_m$ by removing the open stars of the vertices in $P$ (i.e., the vertices in $P$ and edges adjacent to them). 
Note that $K_m$ naturally identifies with a compact subcomplex of both $\widehat K_m$ and $K$.
Moreover, since we are working with an exhausting sequence of peripheral systems, every vertex of $K$ is eventually a vertex of $K_m$.

By the results in \S\ref{sec: compact conical}, for each $m$ we have a label
$\widehat R_m :\widehat K_m^{(0)}\to (0,+\infty]$.
Note that $\widehat R_m$ can take the value $+\infty$ only at points $p_i\in P$ where $\theta_i=0$.
So, we can restrict $\widehat R_m$ to a bounded function on $K_m$ and then extend this to the function on $K$ given by
$$R_m: K^{(0)}\to [0,+\infty), \quad R_m(v)=  \left\lbrace  \begin{array}{ll}
  \widehat R_m (v)   & \textrm{ if } v\in  K_m^{(0)}  \\
  0   &  \textrm{ if } v \not\in K_m^{(0)}
\end{array} \right.$$
By definition, $R_m$ is a non-negative and bounded function, which agrees with $\widehat R_m$ on the compact subcomplex $K_m$, and vanishes elsewhere.
(Note that $R_m$ is not a label for $K$.)
Moreover, by \eqref{item:upper} in Proposition~\ref{prop: radii estimates}, the sequence of functions $R_m$ is uniformly bounded, as all radii $\widehat R_m (v)$ can be bounded above by a constant $H$ only depending on the angles $\Theta$ and the degree $d$ of $K$.
In particular, these labels live in a compact space of functions, so, up to passing to a subsequence, we can assume that $R_m: K^{(0)}\to [0,H]$ converges to a function
$$R:K^{(0)}\to [0,H], \quad R(v)=\lim_{m\to +\infty} R_m(v).$$
 
(Recall from Remark~\ref{rem:convergence} that this limit may not be unique in general, although it is in some special cases, see \S\ref{sec: prescribe conformal}.)
In order for this function to be considered  a legitimate label (hence define a packed conical hyperbolic metric as in \S\ref{sec:labels}) we need to check that $R>0$.
Note that the next statement does not depend on the choice of the subsequence.


\begin{proposition}
If $\theta_i<\pi$ for all $i=1,\dots,n$, then for all $v\in K^{(0)}$ we have $R(v)>0$.
\end{proposition}
\begin{proof}
Let $v\in K^{(0)}$. 
Since $K$ is a triangulation of $S$, its vertices are not punctures, i.e., $v\not \in P$.
Notice that for $m$ large enough, we have $v\in \widehat K_m^{(0)}$ and $R_m(v)=\widehat R_m(v)>0$.

Assume by contradiction $R_m(v)\to 0$.
Let $w \in K^{(0)}$ be an adjacent vertex.
Without loss of generality, let $m$ be large enough so that  both $v$ and $w$ belong to $\widehat K_m$.
By Proposition~\ref{prop: radii estimates} we have that 
$$ E_D \leq  \frac{\widehat R_m(w)}{\widehat R_m(v)} \leq  \frac{1}{E_D},$$
Since $R_m(v)=\widehat R_m(v)$ and $R_m(v)\to 0$, we get $\widehat R_m(v)\to 0$
and therefore $R_m(w)=\widehat R_m(v)\to 0$ too.
Iterating the argument shows that $R_m(u)\to 0$ for all vertices $u\in K^{(0)}$, i.e., $R\equiv 0$.
We will show that this leads to a contradiction.

For an edge-path $\gamma$ in $K^{(1)}$ we let $|\gamma|$ be its combinatorial length (i.e., the number of edges) and $\ell_m(\gamma)$  be its length in the hyperbolic metric $h_m$.
Let $v_0,\dots, v_{|\gamma|}$ be the vertices along $\gamma$.
When $m$ is large enough they are all contained in $\widehat K_m^{(0)}$, and we have that
    $$\ell_m(\gamma)=\sum_{i=0}^{|\gamma|} 2 R_m(v_i).$$
    Therefore, the condition that $R_m\to 0$ implies that $\ell_m(\gamma)\to 0$.
We show that this leads to a contradiction by distinguishing two cases.

First, if $S$ is a thrice-punctured sphere, then the sequence of metrics $h_m$ is constant, see Remark~\ref{rem: 3p sphere conformal}.
Let $\gamma$ be an essential and non-peripheral edge-loop.
(Here, a curve is essential if it is not nullhomotopic, and non-peripheral if it is not homotopic to a puncture.)
Let $\bar \gamma$ be the geodesic representative of the free homotopy class of $\gamma$. 
Since $\gamma$ is essential and non-peripheral, $\bar \gamma$ is not reduced to a single puncture, so its length $\ell_m(\bar\gamma)$ is some positive constant $L>0$; note that $L$ is independent of $m$ because the sequence of metrics $h_m$ is constant.
It follows that
$\ell_m(\gamma)\geq \ell_m(\bar\gamma)=L>0$, so $\ell_m(\gamma)$ cannot go to zero.

Now suppose $S$ is not a thrice-punctured sphere.
 Pick two simple edge-loops $\alpha$ and $\beta$ in $K^{(1)}$ which are essential non-peripheral and intersect transversely.
 Let $\alpha_m$ and $\beta_m$ be the geodesic representatives of $\alpha$ and $\beta$ in the metric $h_m$.
 Then $\alpha_m$ and $\beta_m$ are simple closed geodesics intersecting transversally and not going through cone points, see \cite[Theorem 5.1]{TWZ06} for details. 
 As before, since $R_m\to 0$, we have  that the lengths of $\alpha$ and $\beta$ go to zero, therefore the length of their geodesic representatives $\alpha_m$ and $\beta_m$ also goes to zero.
 However,  the collar lemma for hyperbolic surfaces with cones smaller than $\pi$ (see \cite[Corollary 4]{DP07})    implies  that
    $$\sinh \left(\frac{\ell_m(\alpha_m)}{2}\right) \cdot \sinh \left(\frac{\ell_m(\beta_m)}{2} \right)\geq \cos \left( \frac{\max \Theta}{2} \right).$$
This leads to the desired contradiction.
\end{proof}

We are now ready to complete the proof of Theorem~\ref{thm: main convergence}. Recall that in order to simplify the exposition, cusps are regarded as conical singularities with angle $\theta=0$.

\begin{proof}[Proof of Theorem~\ref{thm: main convergence}]
    The function $R$ takes values in $(0,+\infty)$, hence defines a label for $K$.
    We get the associated hyperbolic metric $h_{K,\Theta}$  by pasting triangles as in \S\ref{sec:labels}.
    Since $R$ is finite at every vertex, 
    this metric has angle $2\pi$ at every vertex of $K$.
    Note that by construction $h_{K,\Theta}$  is approximated by the metrics $h_m$. More precisely, note that in the construction in \S\ref{sec:labels}, the hyperbolic metric depends continuously on the defining label, in the sense that nearby labels define nearby metrics, in the space of hyperbolic metrics with conical singularities; see \cite{BS04}.
    In particular, the metric defined by the limit of a sequence of labels is the limit of the metrics defined by the labels in the sequence.
    Moreover, all the metrics $h_m$ have the desired cone angles at the points in $P$; here we are considering $h_m$ as a metric on the closed surface $\bar S$ with conical singularities at the points in $P$.
    Since having prescribed angles at a prescribed collection of points is a closed condition in the space of hyperbolic metrics with conical singularities, the limiting metric $h_{K,\Theta}$ will have the desired conical singularities as well. 
\end{proof}


\section{Prescribing the conformal class}\label{sec: prescribe conformal}
\noindent In many cases, we can arrange for the sequence of hyperbolic metrics $h_m$ to be even constant. 
Since they are hyperbolic metrics on the same topological surface and have the same conical singularities, it is enough to show that they are in the same conformal class; see Remark~\ref{rem: conical unif}.
For an easy example, when $S$ is a thrice-punctured sphere there is only one conformal class, hence all the $h_m$ are automatically isometric to each other; see Remark~\ref{rem: 3p sphere conformal}.
Examples in higher genus can be obtained by taking suitable normal branched covers of the thrice-punctured sphere, as we now describe.

A Riemann surface $X$ of genus $g\geq 2$
is called \textit{\gb}  if it satisfies any of the following  equivalent conditions (see \cite{WO97} and \cite{FR19}):
\begin{enumerate}
  \item $X/\aut X = \cp$ and the quotient map $\pi:X\to X/\aut X$ is ramified exactly on three points.
  \item $X$ is an isolated local maximum for the function $Y\mapsto \#\aut Y$.
  \item $X\cong\hh^2/\Gamma$, where $\Gamma$ is a cocompact torsion-free Fuchsian group whose normalizer in $\pslr$ is a hyperbolic triangle group.
  \item The Fuchsian uniformization of $X$ is the unique complex projective structure on $X$ such that every automorphism of $X$ is complex projective.
\end{enumerate}
Note that if $X$ is \gb \ then the quotient map $\pi:X\to X/\aut X$ is both a Bely\u{\i} function  and a (branched) Galois cover.
These surfaces, also known as \textit{quasiplatonic surfaces} (see \cite{SI}), play a key role in the theory of dessins d'enfants.
We refer to \cite[Chapter 5]{JW} for more details and background about their theory.

\begin{example}\label{ex:gb}
Classical examples of \gb \ surfaces are given by Hurwitz surfaces. 
Hurwitz surfaces do not occur in every genus: for instance, they do not exist in genus $2$.
On the other hand, for any genus $g\geq 2$,  one can obtain examples of \gb \ surfaces by looking at the hyperelliptic curves $y^2=x^m-1$ in $\cc^2$, see \cite[\S 5.3]{JW}. 
For analogues in genus $g=0$ one should consider punctured spheres with very symmetric configurations of punctures, such as the one given by the regular tetrahedron. In genus $g=1$ one can take the two complex tori with extra automorphisms, namely the ones arising from the tilings of the Euclidean plane by unit squares and equilateral triangles.
    The following arguments can be adapted to these cases, and we leave the details to the interested reader.
\end{example}

\begin{remark}
   Riemann surfaces that admit a Bely\u{\i} function  are exactly the ones that can be defined over $\overline{\qq}$ as algebraic curves.
   They form a countable dense subset of the moduli space. 
   On the other hand, \gb \ surfaces of a fixed genus $g$ form a finite set;  see \cite{POP} and \cite{SPW}.
\end{remark}


Let $X$ be a \gb \ surface of genus $g \geq 2$. Then $X$ carries a standard triangulation $\widehat K$, obtained by pulling back a decomposition of $X/\aut X$ in two triangles whose vertices are the three ramification points of $\pi:X\to X/\aut X$.
In particular, every vertex of $\widehat K^{(0)}$ lies in the preimage of one of the ramification points.
Note that this is well-defined up to isotopy, and that $\aut X$ is a finite group.

\begin{theorem}\label{thm:failure KMT finite}
    Let $\bar S$ be a closed surface of genus $g\geq 2$, let $X$ be a \gb \ complex structure on it, and let $\widehat K$ be the associated triangulation of $\bar S$.
    Then there exists $P=\{p_1,\dots,p_n\} \subseteq  \widehat K^{(0)}$ such that 
    for any $\theta \in [0,\pi)$ 
    there exists a hyperbolic metric 
    in the conformal class of $X$
     which has cone angle $\theta$ at $p_i$ for all $i=1,\dots,n$
    and supports a circle packing in the combinatorics of $\widehat K$.
\end{theorem}
\begin{proof}
    By construction, $\aut X$ acts simplicially on $\widehat K$.
    Let $P=\{p_1,\dots,p_n\}$ be any $\aut X$-invariant collection of vertices of $\widehat K$ (e.g., the orbit of a single vertex).
    Consider the vector of cone angles $\Theta=(\theta,\dots,\theta)$, i.e., assign angle $\theta_i=\theta$ to each $p_i$.
    Since $g\geq 2$, we have that $\Theta$ automatically satisfies \eqref{eq: orbchar}.
    
    Let $h_{\widehat K,\Theta}$ be the hyperbolic metric on $\bar S$ from Theorem~\ref{thm:combinatorial conical uniformization}, and let $X_{\widehat K, \Theta}$ be the underlying conformal structure.
    Note that the isometry type of each triangle is completely determined by the combinatorics of $\widehat K$ and the angle vector $\Theta$; see \S\ref{sec:labels}.
    Since $\Theta$ is clearly invariant under the action of $\aut X$,  by Remark~\ref{rem:invariance} we have that $\aut X$ acts isometrically on $h_{\widehat K,\Theta}$, hence conformally on $X_{\widehat K, \Theta}$.
Since $X,X_{\widehat K,\Theta}$ are complex structures on the same surface $\bar S$, and since $X$ is \gb, we have that topologically $X_{\widehat K,\Theta} / \aut{X} = X/\aut X = \cp$, with quotient map branching over three points.
This implies that $X_{\widehat K,\Theta}$ is \gb.
Actually, $X_{\widehat K,\Theta}=X$, because the two  complex structures are given by the same branched cover of $\cp$ over three points, and all triples of points in $\cp$ are conformally equivalent.
In particular,  $h_{\widehat K,\Theta}$ is in the conformal class $X$, independently of $\Theta$.     
\end{proof}

An analogous result holds for infinite triangulations of a punctured surface.

\begin{theorem}\label{thm:failure KMT infinite}
Let $\bar S$ be a closed surface of genus $g\geq 2$, and let $X$ be a \gb \ complex structure on it.
Then there exists $P=\{p_1,\dots,p_n\} \subseteq \bar S$ and a triangulation $K$ of $S=\bar S\setminus P$ such that 
for any $\theta \in [0,\pi)^n $ 
there exists a hyperbolic metric 
in the conformal class of $X$
     which has cone angle $\theta$ at $p_i$ for all $i=1,\dots,n$
    and supports a circle packing in the combinatorics of $K$.
\end{theorem}
  
\begin{proof}
Let $K_0$ be the standard triangulation of $\bar S$ induced by the \gb \ structure $X$.
Let $K_0'$ be the first barycentric subdivision of $K_0$.
By construction, $\aut X$ acts simplicially on $K_0$ and on $K_0'$.
Let $P=\{p_1,\dots,p_n\}$ be any $\aut X$-invariant collection of vertices of $K_0$  (e.g., the orbit of a single vertex), regarded as vertices of $K_0'$. 
    
    The stars of vertices of $P$ in $K_0'$ are a collection of triangulated disks with disjoint interiors.
    Subdivide $K_0'$ as follows:
    if $v\in P$, then subdivide triangles in $\st{v}{K_0'}$ as shown on the left in Figure~\ref{fig:subdivision}. Do not subdivide any other triangle. (The right-hand side of picture shows the effect of applying this subdivision twice to the disk on the left-hand side.)
    The star of $v$ in this subdivision is isomorphic to the star of $v$ in $K_0'$, so we can iterate the process.
    Since stars of vertices of $P$ have disjoint interiors, this can be performed simultaneously at all vertices of $P$.
    We obtain a sequence of triangulations $K_m$ of $\bar S$.
    In the limit we get a triangulation $K$ of $S$, in which all vertices of $P$ have become punctures.
    
    Note that by construction $K$ comes with a natural choice of a peripheral system with respect to $P$, 
    given by $\lk{v}{K_m}$ for $v\in P$.
    Moreover, $K$ has bounded degree and $\aut X$ still acts simplicially on $K_m$ and $K$.
    Once again, consider the constant vector $\Theta=(\theta,\dots,\theta)$, which  is $\aut X$-invariant and automatically satisfies \eqref{eq: orbchar}, since $g\geq 2$.
    Then we get hyperbolic metrics $h_m=h_{K_m,\Theta}$ as in \S\ref{sec: circle pack infinite}.
    Moreover, arguing with Remark~\ref{rem:invariance} as in the proof of Theorem~\ref{thm:failure KMT finite} we have that $\aut X$ acts by isometries of these metrics, hence all the metrics $h_m$ are in the same conformal class $X$, independently of $\Theta$.
    In the limit for $m \to \infty$, we get a metric $h$ in the conformal class $X$, with the prescribed angles $\Theta$ on $P$, and supporting a circle packing in the combinatorics of $K$.
    \end{proof}
    
\begin{figure}[h]
    \centering
    \includegraphics[width=.75\textwidth]{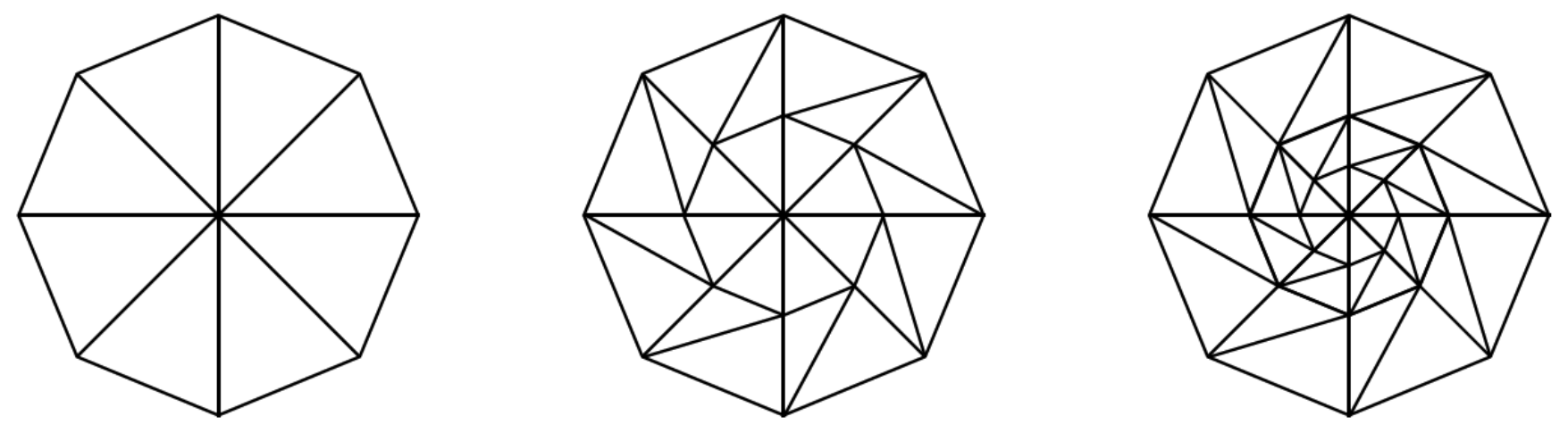}
    \caption{Two steps in the subdivision in the proof of Theorem~\ref{thm:failure KMT infinite}.}
    \label{fig:subdivision}
\end{figure}

Note that the subdivision procedure in Figure~\ref{fig:subdivision} provides an infinite triangulation $K$ in which each end is accumulated by a sequence of edge-loops of constant combinatorial length. 
In particular, it follows that $K$ has parabolic ends, and therefore the packing fills the surface; see \S\ref{sec:triangulations} and \cite{HS95}.

\begin{remark}
    In the above statements, the minimal collection of punctures $P$ is the $\aut X$ orbit of a single vertex, whose size is naively bounded just in terms of the genus by $n\leq |\aut X|\leq 84(g-1)$, thanks to the classical Hurwitz  Theorem; see \cite{HU}.
    More generally, the set of punctures $P$ can consist of all the preimages of one, two, or all three ramification points of $\pi:X\to X/\aut X$.
\end{remark}

We conclude by showing that the non-compact analogue of the conjecture by Kojima-Mizushima-Tan (see \cite{KMT03,KMT06}) does not hold.
Recall the following setup from \S\ref{sec:intro}.
For a surface $S$, we denote by $\projteich S$  the space of projective structures with  conical singularities or cusps.
It contains the space of hyperbolic metrics with conical singularities or cusps.
Given a triangulation $K$, we denote by $\projlocus SK$ the subspace of $\projteich S$ consisting of projective structures supporting a circle packing in the combinatorics of $K$, and we consider the natural forgetful map
$$\kmt K:\mathcal C(S,K)\to \teich S .$$
The following results are obtained from Theorem~\ref{thm:failure KMT finite} and Theorem~\ref{thm:failure KMT infinite} respectively by taking $X$ to be a \gb \  complex structure, and choosing angles $\theta<\pi$.

\begin{corollary}
Let $\bar S$ be a closed surface of genus $g\geq 2$.
There exists a triangulation $K$ of $\bar S$ and  a complex structure $X$ on $\bar S$ such that $\kmt K^{-1}(X)$ is uncountable.
\end{corollary}

\begin{corollary}
Let $\bar S$ be a closed surface of genus $g\geq 2$.
Then there exists $P=\{p_1,\dots,p_n\} \subseteq \bar S$, a triangulation $K$ of $S=\bar S\setminus P$ and a complex structure on $S$ such that $\kmt K^{-1}(X)$ is uncountable.
\end{corollary}

The difference between the two statements is the following. 
In the former, $K$ is a triangulation of a closed surface and the structures have conical singularities at some of the  vertices of $K$.
In the latter, $K$ is a  triangulation of a non-compact surface of finite type and the structures have conical singularities precisely at ends of $K$; in this case $K$ has infinitely many vertices accumulating into the ends of $K$, and the metrics are smooth at the vertices of $K$.

Note that choosing different values for $\theta$ results in different hyperbolic structures that remain different when regarded as projective structures (e.g., punctures with different cone angle have different peripheral holonomy; see \cite{BBCR21}).

\printbibliography

\end{document}